\documentclass[12pt]{amsart}
\usepackage[utf8]{inputenc}
\usepackage[T1]{fontenc}
\DeclareMathAlphabet{\mathpzc}{OT1}{pzc}{m}{it}
\usepackage{geometry}
\usepackage{comment}

\usepackage{amsfonts}
\usepackage{amssymb}
\usepackage{amsthm}
\usepackage{amsmath}
\usepackage{amscd}
\usepackage[shortlabels]{enumitem}
\usepackage{mathrsfs}
\usepackage{tikz}
\usetikzlibrary{calc,arrows,decorations.pathreplacing}
\usepackage{nicefrac, xfrac}
\usepackage{mathtools,xparse}
\usepackage{subcaption}
\usepackage{graphicx}
\usepackage{blindtext}
\usepackage{hyperref}

\theoremstyle{plain}

\newtheorem{theorem}{Theorem}[section]

\newtheorem{proposition}[theorem]{Proposition}

\theoremstyle{definition}

\newtheorem{remark}[theorem]{Remark}

\newtheorem*{theorem*}{Theorem}

\renewcommand{\phi}{\varphi}
\renewcommand{\epsilon}{\varepsilon}

\newcommand{\vertiii}[1]{{\left\vert\kern-0.25ex\left\vert\kern-0.25ex\left\vert #1
		\right\vert\kern-0.25ex\right\vert\kern-0.25ex\right\vert}}

\newcommand{\eps}{\ensuremath{\epsilon}}

\DeclareSymbolFont{bbold}{U}{bbold}{m}{n}
\DeclareSymbolFontAlphabet{\mathbbold}{bbold}

\DeclareMathOperator{\diam}{diam}

\title[]{Topological synchronisation or a simple attractor?}
\date{\today}
\author{Th\'eophile Caby}
\address{CMUP, Departamento de Matem\`atica, Faculdade de Ciências, Universidade do Porto,
Rua do Campo Alegre s/n, 4169007 Porto, Portugal.}
\email{caby.theo@gmail.com}
\author{Michele Gianfelice}
\address{
Dipartimento di Matematica e Informatica
Universit\`a della Calabria
Campus di Arcavacata
Ponte P. Bucci - cubo 30B
I-87036 Arcavacata di Rende}
\email{gianfelice@mat.unical.it}

\author{Beno\^it Saussol}
\address{I2M, Aix Marseille Université,, 13009 Marseille, France}
\email{benoit.saussol@univ-amu.fr}

\author{Sandro Vaienti}
\address{Aix Marseille Université, Université de Toulon, CNRS, CPT, 13009 Marseille, France}
\email{vaienti@cpt.univ-mrs.fr}

\begin{document}
\maketitle
\begin{abstract}
A few recent papers introduced the concept of topological synchronisation. We refer in particular to \cite{TS}, where the theory was illustrated by means of a skew product system, coupling two logistic maps. In this case, we show that the topological synchronisation could be easily explained as the birth of an attractor for increasing values of the coupling strength and the mutual convergence of two marginal empirical measures. Numerical computations based on a careful analysis of the Lyapunov exponents suggest that the attractor supports an absolutely continuous physical measure (acpm). We finally show that for some unimodal maps such acpm exhibit a multifractal structure.
\end{abstract}
\section{Introduction}
The recent paper \cite{TS}, which also garnered some press attention \cite{PP1, PP2},  introduced the concept of topological synchronisation which occurs when, in a dynamical system, it is possible to identify two or more attractors which become very similar when the system evolves. This situation is for instance met in coupled lattice map, where each site of the lattice brings its own attractor.  It is written in \cite{TS} that \emph{during the gradual process of topological
adjustment in phase space, the multifractal structures of each strange attractor of the two coupled
oscillators continuously converge, taking a similar form, until complete topological synchronization
ensues.} As an example of this process of synchronisations, the authors in \cite{TS} studied  a skew system whose base is a logistic (master) map $T$ of the interval $[-1,1]$ and the other map (the slave), is another logistic map on the same interval which is coupled with the master in a convex way in order to be confined to the interval $[-1,1]$. As an indicator of the {\em closeness} of the attractors of the master and slave maps when the coupling strength, say $k,$ increases, the authors in \cite{TS} used the spectrum $D_q$ of generalized dimensions. They showed in particular the interesting phenomenon, which they called the {\em zipper effect}, where the dimensions begin to synchronise at negative $q$, with low values of $k$ before becoming similar, for positive values of $q$, when $k$ arrives at the threshold of complete synchronisation of the attractors. They interpreted this fact by saying that \emph{the road to complete synchronization starts at low coupling with topological synchronization of the
sparse areas in the attractor and continues with topological synchronizations of much more dense areas in the
attractor until complete topological synchronization is reached for high enough coupling}.\\

The object of our note is to show that, in the case of the skew system where the master and the slave map are both in the logistic family, if we denote with $\{ x_n\}_{n \geq 0}$ the trajectory of the master system and with $\{ y_n\}_{n \geq 0}$ that of the slave system,  the topological synchronisation is easily interpreted as the presence of an invariant set in the neighborhood of the diagonal of the square $[-1,1]^{2}$ to which $\{ x_n\}_{n \geq 0}$ and $\{ y_n\}_{n \geq 0}$ converge when the coupling strength tends to $1,$ in the sense of (\ref{ff}). Moreover, we show that the empirical measure computed along the trajectories of the slave system approaches, in the limit as the number of the iterations tends to infinity, the physical measure of the master map. We  compute numerically the Lyapunov exponent of the master map $T$ and we show that it is positive for the parameter values considered in \cite{TS}, which implies that the attractor in the master space is a finite union of intervals. We therefore discuss the real occurrence of a multifractal spectrum for the empirical measure of unimodal maps. We prove the existence of a non-trivial multifractal spectrum for the Benedicks-Carleson type maps investigated in \cite{baladi} and where the invariant density has at most countably many poles. We show that the generalized dimensions are constant and equal to $1$ for $q<2,$ and so in particular for negative $q$ and this explains easily the presence of the zipper effect. We finally  give a toy-model example of an invariant density generating a multifractal spectrum on a Cantor set of poles.

\section{The attractor} The skew system studied in \cite{TS} is defined on the square $[-1,1]^{2}$ and has the form for $0<k<1$
\begin{equation}\label{00}
\begin{cases}
x_{n+1}= T_1(x_n) \\
y_{n+1}=(1-k)T_2(y_n)+k T_1(x_n),
\end{cases}
\end{equation}
where $T_1$ and $T_2$ are two maps of the interval $[-1,1]$ into itself. Set
$$
\Delta_n:=|x_n-y_n|;
$$
it is immediate to see that for any $n\ge 1:$
\begin{equation}\label{MB}
\Delta_n\le 2(1-k)\sup_{i=1,2}\sup_{x\in [-1,1]}|T_i(x)|,
\end{equation}
and therefore the sequences $x_n, y_n$ approach each other when $k\rightarrow 1.$
We now specialize to the example investigated in \cite{TS} and show how to improve the previous bound. The skew system now reads:
\begin{equation}\label{00}
\begin{cases}
x_{n+1}= T(x_n)=c_1(1-2x_n^2) \\
y_{n+1}=(1-k)c_2(1-2y_n^2)+c_1k(1-2x_n^2).
\end{cases}
\end{equation}

We have
$$
\Delta_{n+1}=|(1-k)c_1(1-2x_n^2)-(1-k)c_2(1-2y_n^2)|\le
$$
$$
(1-k)|c_1-2c_1x_n^2-c_2+2c_2y_n^2|.
$$
We add and subtract the term $2c_1y_n^2$ and we easily obtain
$$
\Delta_{n+1}\le (1-k)|(c_1-c_2)(1-2y_n^2)+2c_1(y_n^2-x_n^2)|.
$$
We now put $\Delta_c:=|c_1-c_2|.$  Since $x_n, y_n$ are in the 
 interval $[-1,1]$, we have
$$
\Delta_{n+1}\le (1-k) \Delta_c|1-2y_n^2|+ 4c_1(1-k)\Delta_n\le
 (1-k) \Delta_c+ 4c_1(1-k)\Delta_n.
$$
We now iterate it and we finally get
$$
\Delta_{n+1}\le (1-k)\Delta_c \sum_{l=0}^n (1-k)^l(4c_1)^l+ 
(1-k)^{n+1}(4c_1)^{n+1}.
$$
We then
require 
\begin{equation}\label{e1}
k>1-\frac{1}{4c_1},
\end{equation}
and we define the quantity
\begin{equation}\label{2i}
W_{\infty}(k):=\Delta_c(1-k)\frac
{1}{1-(1-k)4c_1}, \  \text{such that} \ \lim_{k\rightarrow 1} W_{\infty}(k)=0.
\end{equation}

By sending $n\rightarrow \infty$ we finally get 
\begin{equation}\label{ff}
\limsup_{n\rightarrow \infty}\Delta_n\le W_{\infty}(k).
\end{equation}

We now use the following values taken in \cite{TS}:
$$
c_1=0.89, \ c_2=0.8373351.
$$
First of all we note that with these values (\ref{e1}) gives
$
k>0.72,
$
which is consistent with what was used in \cite{TS}. As in the latter we now
take
$
k=0.9,
$
which is the value where, according to \cite{TS}, the system reaches complete topological
synchronization.
By substituting into $W_{\infty}(k)$ we  get
$$
W_{\infty}(0.9)=0.0082,
$$
which implies that the projections $x_n$ and $y_n$ are really very
close. The bound (\ref{MB}) instead gives, still for $k=0.9,$
$$
\sup_{n\ge 1} |\Delta_n|\le (1-k)[c_1+c_2]=0.17273351.
$$
\section{The measures}
In order to justify the closeness of the asymptotic behaviors  of the master and slave dynamics, the paper \cite{TS} uses the spectrum of the generalized dimensions. These dimensions are defined in terms of a probability measure, see, e.g., \cite{DQ} and \cite{pesin, dem} for a rigorous treatment. Roughly speaking, if $\mu$  denotes a probability measure, and $B(x,r)$ a  ball of center $x$ and radius $r$ on the phase space $M$, the generalized dimensions $D_q$ are defined by the following scaling of the correlation integral
\begin{equation}\label{CCII}
\int_M \mu(B(x,r))^{q-1}d\mu \sim r^{D_q(q-1)}, \ r\rightarrow 0.
\end{equation}
The importance of the generalized dimension is that in several cases, see \cite{pesin} and the references therein for rigorous results, if we denote
by 
$$
d_\mu(x):=\lim_{r\rightarrow 0}\frac{1}{\log r}\log \mu(B(x,r)),
$$
the local dimension of the measure $\mu$ at the point $x$, provided the limit exists, then 
\begin{equation}\label{DDII}
D_q(q-1)=\inf_{\alpha}\{q\alpha-f(\alpha)\},
\end{equation}
where $f(\alpha)$ denotes the Hausdorff dimension of the set of points for which $d_\mu(x)=\alpha\footnote{It is worth noticing that in the next section we will compute the $D_q$ spectrum in a few cases by using the characterization (\ref{DDII}) and {\em not} the definition (\ref{CCII}) in terms of the correlation integral.}.$\\
The master map has several invariant probability measures;  we pick one,  namely the {\em physical measure} $\mu$ which is given by the weak limit of the probability measures
\begin{equation}\label{rr}
\mu_n:=\frac1n \sum_{i=0}^{n-1}\delta_{T^ix},
\end{equation}
where $x$ is chosen Lebesgue almost everywhere on the unit interval (see for instance  \cite{MS} Chapter V.1).
In the following we will forget about the initial condition $x,$ provided it is taken Lebesgue almost everywhere, and simply write
$
x_i=T^i(x).
$
The {\em slave} sequence $\{y_n\}_{n\ge 0}$ could be seen as a non-autonomous, or {\em sequential},  dynamical system and it is not clear what probability measure  we should associate to it. We argue that in \cite{TS} the authors used the sequence of probability measures
$$
\nu_n:=\frac1n \sum_{i=0}^{n-1}\delta_{y_i},
$$
where $y_i$ is the point associated to $x_i$ in (\ref{00}). We call $\mu_n$ and $\nu_n$ the {\em empirical measures}. \\ There are now two questions: (i) does the sequence $\nu_n$ converge weakly? and, in the affirmative case, (ii) is that weak limit point equal to $\mu?$  This is in fact what the numerical simulations on the generalized dimensions seem to indicate in \cite{TS}.\footnote{We point out, however, that it is in general not enough to have a weak convergence of the measures to ensure the convergence of the $D_q$ spectrum. Suppose for instance that the master system has an absolutely continuous invariant measure $d\mu(x)=h(x) dx$ and that, for $k$ close enough to 1, so does the measure of the slave system $d\mu_k(x)=h_k(x) dx$. If $h(x)\sim_{x_0} \text{const} |x-x_0|^{\alpha},$ with $-1<\alpha<0$  as, for instance, it is the case for some quadratic map along the orbit of the critical point, then the local dimension of $\mu$ at $x_0$ is $\alpha +1<1$ and it is easily seen (see the detailed computations in the section 4.2) that the $D_q$ spectrum is not constant. Moreover, if we further assume that, for all $k<1$, $h_k$ is a piecewise constant function converging in $L^1$ to $h$, it is easy to see that the $D_q$ spectrum for $\mu_k$ is constant equal to 1 for all $k<1$, so that there is no convergence to the spectrum of the master map.} \\

To study weak convergence, we have to integrate the probability measures against continuous functions defined on the  interval $[-1,1]$. Let $f$ be one of this function; since it is also uniformly continuous, given $\eps>0,$ call $\delta_{\eps}$ the quantity such that $|f(x)-f(y)|<\frac{\eps}{2},$ when $|x-y|<\delta_{\eps}.$ Let $k_{\eps} \in (0,1)$ such that
\begin{equation}
\label{fr}2W_{\infty}(k_{\eps})<\delta_{\eps}.
\end{equation}
For values of $k$ such that $k_{\eps}<k<1$, we define $n_{k,\eps}$ as
$$
(1-k)^{n_{k, \eps}+1}(4c_1)^{n_{k, \eps}+1}\le W_{\infty}(k) ,
$$
such that for all $n> n_{k,\eps}, \Delta_n\le \delta_{\eps}.$ By weak-compactness there will be a subsequence $n_l$ for which $(\nu_{n_l})_{l\ge 1}$ will converge weakly to a probability measure $\mu^*$. Then for any continuous function $f$ on the unit interval and for $n_l$ sufficiently large, say $n_l>n^*$, we have that $|\frac{1}{n_l}\sum_{i=0}^{n_l-1}f(y_i)-\mu^*(f)|\le \eps/2.$ Then
$$
|\mu^*(f)-\mu(f)|\le \left|\frac{1}{n_l}\sum_{i=0}^{n_l-1}f(y_i)-\mu^*(f)\right|+\left|\frac{1}{n_l}\sum_{i=0}^{n_l-1}f(y_i)-\mu(f)\right|
$$
We now estimate the second piece on the right hand side:
$$
\frac{1}{n_l}\sum_{i=0}^{n_l-1}f(y_i)-\mu(f)=
\frac{1}{n_l}\sum_{i=0}^{n_l-1}f(y_i)-\mu(f)+\frac{1}{n_l}\sum_{i=0}^{n_l-1}f(x_i)-\frac{1}{n_l}\sum_{i=0}^{n_l-1}f(x_i).
$$
Now, for $n_l\ge n_{k, \eps}+2$, consider the difference
$$
\frac{1}{n_l}\sum_{i=0}^{n_l-1}[f(y_i)-f(x_i)]=\frac{1}{n_l}\sum_{i=0}^{n_{k, \eps}}[f(y_i)-f(x_i)]+\frac{1}{n_l}\sum_{i=n_{k, \eps}+1}^{n_l-1}[f(y_i)-f(x_i)].
$$
By exploiting the uniform continuity of $f$ on the unit interval we have
$$
\left|\frac{1}{n_l}\sum_{i=n_{k, \eps}+1}^{n_l-1}[f(y_i)-f(x_i)]\right|\le \frac{n_l-n_{k, \eps}-2}{n_l}\eps/2.
$$
The other piece gives
$$
\left|\frac{1}{n_l}\sum_{i=0}^{n_{k, \eps}}[f(y_i)-f(x_i)]\right|\le 2\max|f| \frac{n_{k+\eps}+1}{n_l}.
$$
By sending $l\rightarrow \infty,$ we finally get that $|\mu^*(f)-\mu(f)|\le \eps,$ and this result is independent of the subsequence we choose. We thus have

\begin{proposition}
(i) For any $\eps>0,$ let $k_{\eps} \in (0,1)$ be given as in (\ref{fr}); then for all $k$ such that $k_{\eps}<k<1,$ we have
$$
\left|\limsup_{n\rightarrow \infty}\frac1n\sum_{i=0}^{n-1}f(y_i)-\mu(f)\right|\le \eps.
$$
(ii)  As a consequence we get:
$$
\inf_{0<k<1}\left|\limsup_{n\rightarrow \infty}\frac1n\sum_{i=0}^{n-1}f(y_i)-\mu(f)\right|=0.
$$
\end{proposition}

This is the best  result we could get without further information on the system and it justifies the numerical evidence that the empirical measures constructed along the $x$ and $y$ axis become very close to each other when $k\rightarrow 1$.\\

\section{The nature of the master's physical invariant measure}

We said above that $\mu,$ the invariant measure for the master map $T,$ is a physical measure; the paper \cite{TS} claims that such a measure has a multifractal structure for the prescribed values of $c_1,$ where {\em the master  has a dense strange attractor}, [ibid]. Before exploring and commenting such a  possibility, we should remind a few important properties of the quadratic maps: first that they usually depend on a parameter, in our case $c$ since the map in \cite{TS} is of the form
\begin{equation}\label{map}
[-1,1] \ni x \longmapsto T(x)=c(1-2x^2) \in [-1,1],
\end{equation}
with $0<c\le 1$.
We refer in particular to the nice review paper by Thunberg \cite{th}, which contains a clear and exhaustive list of all the relevant results on unimodal maps and a rich bibliography. First of all, we define the attractor $\Omega_c$ of the map $T$ as the unique   set of accumulation
points of the orbit of the point $x,$ whenever this point is chosen Lebesgue almost everywhere.  Then it is well known, see \cite{L} or Theorem 6 in \cite{th}, that for our kind of logistic maps, the attractor could be of three types:\\ (1) an attracting periodic orbit;
(2) a Cantor set of measure zero;
(3) a finite union of intervals with a dense orbit.\\
Still in the quadratic case, we could classify the preceding three different types of attractors in terms of the set of parameters $c$. Following section 2.2 in \cite{th} we have:

(1) $\mathcal{P} := \{c \in \mathbb{R} : \  \Omega_c \ \text{ is a periodic cycle}\}$ is open and dense in parameter space and consists of countably infinitely many nontrivial intervals.

(2) $\mathcal{C} := \{c \in \mathbb{R} : \Omega_c \ \text{ is a Cantor set}\}$  is a completely disconnected set of Lebesgue
measure zero.

(3) $\mathcal{I} := \{c \in \mathbb{R} : \Omega_c \ \text{ is a union of intervals}\}$  is a completely disconnected set of positive
Lebesgue measure.

The physical measures, constructed according to the prescription (\ref{rr}) exist and are parametrized by $c$ in the following way:

(1) If $c\in \mathcal{P}$, the physical  measure  consists of normalized point masses on the periodic cycle $\Omega_c$.

(2) If $c\in \mathcal{C}$, the support of the physical measure  equals the Cantor attractor $\Omega_c$, and it is singular
with respect to Lebesgue measure.

(3) (a) There is a full-measure subset $\mathcal{S}\subset \mathcal{I}$ such that for all $c\in \mathcal{I}$, the physical measure is absolutely continuous with respect to Lebesgue measure and its support  equals the interval
attractor $\Omega_c$.\\
(b) There are uncountably many parameters in $\mathcal{I}\setminus\mathcal{C}$ for  which the physical measure may fail to exist.\\  

We now dispose of a very efficient   numerical test to determine the nature of  a physical  measure. It is based on the following two rigorous results:

(i) the first says that if $T$ has a non-flat critical point, as in our case, and   it admits an  absolutely continuous invariant probability measure $\mu,$ then  it is the weak-limit of the sequence $\mu_n$  given in (\ref{rr}) and therefore it is a physical measure (see Chapter V.1 in \cite{MS}).

(ii) The second result is taken from the paper \cite{ke}. Let us define the number
$$
\lambda_T:=\limsup_{n\rightarrow \infty}\frac1n\log |(T^n)'(x)|.
$$
This quantity  exists for $x$ chosen $Leb$-almost everywhere and it is strictly positive  if and only if $T$ has an absolutely continuous invariant measure.\\
From the joint use of (i) and (ii) it follows immediately  that if we can show that the sequence
\begin{equation}\label{ror}
\lambda_T^n:=\frac1n\log |(T^n)'(x)|=\frac{1}{n}\sum_{i=0}^{n-1}\log|T'(x_i)|
\end{equation}
has a positive limit for $Leb$-a.e. $x,$ then the sequence of empirical measures $\mu_n$ in (\ref{rr})  converges weakly to an  absolutely continuous invariant probability measure and therefore the attractor $\Omega_c$ will be a finite union of intervals and not a Cantor set. In fig. \ref{bif}, we represent the bifurcation diagram of $T$, and its Lyapunov exponent for different values of parameter $c$. This quantity is non-positive whenever the attractor is a periodic cycle or a Cantor set. We computed in particular the limit of $\lambda_T^n$ for $c=c_1=0.89,$ still called $\lambda_T$, and we got a positive value of $\approx 0.35$, confirming the fact that {\em $\mu$ is not supported on a Cantor set}. We performed the same computation  with $c=c_2$ and, denoting from now on by $\tilde{\mu}$ the associated physical measure, there are strong numerical evidences that  it is again absolutely continuous, with a strictly positive Lyapunov exponent.

\section{Multifractal spectrum for absolutely continuous  measures}\label{mss}

\subsection{Multifractal spectrum for unimodal maps}
Let us summarize: by choosing the parameter $c$ with positive (Lebesgue) probability, we could get a periodic cycle or union of intervals. On the other hand Dirac measures with finitely many masses on the periodic cycles cannot have a multifractal spectrum, since $D_q=0$ for all $q$ in that case. Finally the Lyapunov exponent for $c=c_1$ is positive showing that the attractor could not be a Cantor set and the physical measure will be absolutely continuous. The question is, therefore, if such a measure $\mu$ could exhibit  a multifractal spectrum. Let us consider unimodal maps of Benedicks-Carleson type, which are known to preserve an absolutely continuous invariant measure $\mu$ \cite{Y92}. Let us denote $z_k=f^k(0)$. Under additional assumptions on the dynamics of the critical point\footnote{The map $f$ is of class $C^4$ and it must be:\begin{itemize}
    \item a {\em Collet-Eckmann} $S$-unimodal map verifying $|(f^k)'(f(c))|>\lambda_c^k,$ with $\lambda_c>1,$ $\forall k>H_0,$ where $H_0$ is a constant larger than $1.$
    \item it verifies the {\em Benedicks-Carleson} property: $\exists 0<\gamma<\frac{\log \lambda_c}{14}$ such that $|f^k(c)-c|>e^{-\gamma k}, \forall k>H_0$.
    \end{itemize}}
(see \cite{baladi} for details), their density has the form:
	$$
	h(x) = \psi_0(x) + \sum_{k\ge1} \frac{\phi_k(x)}{\sqrt{|x-z_k|}},
	$$
	with $\psi_0$ a bounded $C^1$ function, and for all $k\ge1$, $\phi_k$ is piece-wise $C^1$, and such that $||\phi_k||_{\infty} \le e^{-a k}$ for some $a>0$.\\
\begin{proposition}\label{bee}
Suppose that $f$ satisfies the hypothesis of Theorem 2.7 in \cite{baladi}, see also footnote 3.  Then, the generalized dimensions spectrum of $\mu$ is given by:

\begin{equation}
D_q=
\begin{cases}
1 \ \text{if} \  q<2\\
\frac{q}{2(q-1)} \ \text{otherwise}. 
\end{cases}
\end{equation}
\end{proposition}

\begin{proof}	
In the following proof the constants $a_j, j=1,2..$ will be independent of $x$ and $r.$ We start by noticing that since $h$ is bounded away from 0 \cite{Y92}, the local dimensions are all smaller than or equal to 1. Now, the measure of a ball centered at $x$ of radius $r$ is given by
	$$
	\mu(B(x,r)):=\int_{x-r}^{x+r}\psi_0(y)dy+\sum_{k\ge1}  \int_{x-r}^{x+r}\frac{\phi_k(y)}{\sqrt{|y-z_k|}}dy.
	$$
\\
	Let us take $m_n=\frac n{2a}$ and $\delta_n=m_n^{-\log n}$. 
	Let $$	\Gamma_n = \{x:\exists k<m_n\textrm{ such that } |x-z_k|<\delta_n \}.$$
	Given $r>0$, we take $n$ the smaller integer such that $r<e^{-n}$. Since the functions $\phi_k$ are bounded, the  integrals in the sum are bounded above by $a_1||\phi_k||_{\infty}\sqrt{r}$ when $|x-z_k|<\delta_n$ and by $a_2||\phi_k||_{\infty}r/\sqrt{\delta_n}$ otherwise , where $a_1,a_2>0$. We get:
	\begin{equation}\label{dd}
	\mu(B(x,r)) \le 2r||\psi_0||_{\infty}+a_1\sqrt{r} \sum_{k:|x-z_k|<\delta_n} ||\phi_k||_{\infty} + a_2\frac{r}{\sqrt{\delta_n} }\sum_{k:|x-z_k|>\delta_n} ||\phi_k||_{\infty}\  
	\end{equation}
	For $x\not\in\Gamma_n$, the first sum starts at least at $m_n$ is therefore at most $a_3e^{-a m_n} \le a_3 \sqrt{r}$. The second geometric sum is bounded par $a_3$.
	Thus there exists $a_4>0$ such that, 
		$$
	\mu(B(x,r)) \le a_4 r + a_4\frac{r}{\sqrt{\delta_n} }.
	$$
	If $x\not\in\Gamma_n$ for $n$ large enough, then $d_\mu(x)=1$, since
	 $-\log \delta_n$ is of order $(\log\log(1/r))^2<\log(1/r)$, so the second term does not affect the dimension. Therefore $d_\mu(x)=1$ in the set $$\displaystyle G= \bigcup_{p}\bigcap_{n>p}\Gamma_{n}^c.$$
	Let $\Gamma = G^c =\bigcap_{p}\bigcup_{n>p}\Gamma_{n}$, the set of $x$ such that there exists an infinity of $n$ such that $x\in\Gamma_n$. $\Gamma $ is covered by the union of balls 
	$$
	\bigcup_n \bigcup_{k<m_n} B(z_k,\delta_n).
	$$
	Now, for all $\epsilon>0$, we have 
	$$
	\sum_n \sum_{k<m_n} \delta_n^\epsilon = \sum_n (\frac n{2a})^{1-\epsilon \log n}<\infty.
	$$
	So the Hausdorff measure $H^\epsilon(\Gamma)$ is finite, which show that $\dim_H(\Gamma)\le\epsilon.$\\

It is easily seen from (\ref{dd}) that for all $x\in[-1,1],$
\begin{equation} 
\mu(B(x,r)) \le a_5\sqrt{r}.
\end{equation}
This shows that the infimum of the local dimensions is larger or equal to $1/2$. On the other hand, since for all $k$, $\phi_k$ is $C^1$ by part, the singularities are of type $|x-z_k|^{-1/2}$. Therefore, if the density admits a singularity at $z_k$,
\begin{equation} 
\mu(B(z_k,r))\ge a_6\sqrt{r}.
\end{equation}
Combining the last two estimates, we get that 
$
d_\mu(z_k)=1/2.
$
We can now compute the generalized dimensions. $D_q(q-1)$ is defined as the Legendre transform of the function $f(\alpha):=d_H\{x; d_\mu(x)=\alpha\},$ where $d_H$ denotes the Hausdorff dimension. In our case, we have $f(1)=1$ and, for all $\alpha<1$, either $f(\alpha)=0$, or $f(\alpha)$ is not defined, so that 
$$D_q=(q-1)^{-1} \inf_\alpha\{q\alpha,q-1\}.$$
Since the density is bounded from below \cite{Y92}, the local dimensions are bounded above by 1, and since $\min(d_\mu)=1/2,$ we obtain our result.

\end{proof}

For our class of quadratic maps (\ref{map}) depending upon the parameter $c$, the assumptions stated in the footnote 3 are satisfied for a positive measure set of values of the parameter $c$, \cite{BBCC, BBVV}. It is therefore plausible, although not certain, that the previous proposition  applies to the physical measure of our master map. As the latter has a  density  bounded away from 0 \cite{ke}, we surely expect its generalized dimensions spectrum to be constant for negative values of $q$ and not differentiable, although its numerical approximation shows a smooth behavior  (see Fig. 5.1 in \cite{LA}, which investigated the fully quadratic map with only one divergent singularity for the density). It is enough for the measure of the slave system to have a density bounded away from 0 as $k$ approaches 1, to yield $D_q=1$ for negative $q$. This would explain the observed {\em zipper effect} described in \cite{TS} for this particular example.\\

\subsection{Densities with a singularity spectrum defined on an interval}
In this section, we construct a density having singularities distributed on a Cantor set, that has a non trivial singularity spectrum. This example does not relate directly to the density of unimodal maps, but is intended to show that non trivial multifractal features can arise from absolutely continuous invariant measures. We sketch the main steps of the proof whose details could be completed with arguments close to those used in the proof of Proposition \ref{bee}. With the symbol $a\asymp b$ we mean that $a$ is bounded from below and above as $C_1a\le b\le C_2a,$ with $C_1, C_2$ two positive constants. \\

Let $T$ be a  $C^2$ piecewise expanding map of the unit circle $I$ with three branches, coded by 0, 1 and 2.  Each  $w\in\{0,1,2\}^\mathbb{N}$ encodes a unique point $x\in I$. We note $u=-\log |T'|$ and $K$ the Cantor set constituted of the set of points whose codes do not contain 1. We denote $P(u|S)$ the topological pressure of $u$ on a set $S$. Let $p$ be the pressure of $u$ on $K$ and $\mu_u$ the Gibbs measure of $u$ on $K$. Note that $p=P(u|K)<P(u|I)=0$.
We fix $\alpha\in (0,-p)$ and define a density with respect to the Lebesgue measure, for $x\notin K$, as 
$$h(w)=\exp[-n (p+ \alpha)],$$ 
where $x$ is coded by $w$ and $n$ is the smallest integer such that $w_n=1$. The measure $\mu$ with density $h$ with respect to Lebesgue is finite,
and has the following properties: for $w$ coding a point in $K$ :
\begin{enumerate}
   \item $\mu(Z_n^w )\asymp \mu_u(Z_n^w) \exp(-\alpha n)$.
\item The diameter of this cylinder $\asymp \exp(n p) \mu_u(Z_n^w)$.
\item  $\mu_u(Z_n^w)\asymp\mu(B(x,r))$, where $n$ is the smallest integer such that $\diam Z_n^w<r$.
\end{enumerate}

This implies that the local dimension for the measure $\mu$ of points in $K$ satisfies
$$
d_\mu(x) =  1-(p+\alpha)\left( \lim_{n\to\infty} \frac1n \sum_{k=0}^{n-1}u(T^k(x))\right)^{-1}.
$$
Therefore the dimension spectrum of the measure $\mu$ is determined by the Lyapunov multifractal spectrum of the map $T$ on $K$. If the latter 
 has a non trivial multifractal spectrum, the set $L_\lambda$ of points $x\in K$ such that the local Lyapunov exponent 
 $$
 \lambda(x) := \lim_{n\to\infty} \frac1n \log (T^n)'(x)=\lambda
 $$
 has a Hausdorff dimension $g(\lambda)$ which is non trivial in an interval of values of $\lambda$, and for these points the local dimension of the measure $\mu$ is exactly $$1+\frac{p+\alpha}{\lambda}<1.$$

We obtain
$$f(1+\frac{p+\alpha}{\lambda})=g(\lambda),$$ for an interval of values of $\lambda$.\\

A tractable example is given by a map which is affine on the branches $0$ and $2$, with different slopes $e^{\lambda_0}$ and $e^{\lambda_2}$. It is known there \cite{pesin} that the multifractal set $L_\lambda$ carries a Bernoulli measure of full dimension $g(\lambda)$, allowing an explicit computation.\\

Note that in this example, the maximum of the local dimensions is 1, and is achieved for the points $x\in I\setminus K$. As a consequence, we have $f(1)=1$, so that $D_q$ is constant equal to 1 for $q\le 0$. On the other hand, for $q$ large enough, $D_q$ can be computed from the generalized dimensions of the Gibbs measure $\mu_u$ on $K$, and exhibits a smooth non trivial behavior whenever $\mu_u$ does. We do not know whether it is possible to construct a density exhibiting a $D_q$ spectrum that is smooth and whose derivative does not vanish in $\mathbb{R}$.

\section{A random analog}
Our proposition 3.1 suggests that the sequence of empirical measures $\nu_n$ for the slave non-autonomous evolution converges weakly to $\mu.$ Such an evolution could be understood in another way. Consider the logistic master map $T;$ at each step $x_n$ we now add a number $(1-k)\omega_n$, where $\omega_n \in [-1,1]$  is {\em taken with the probability distribution given by the invariant measure $\tilde{\mu}$ of the master map $\tilde{T}(x)=c_2(1-2x^2),$} see above. Suppose moreover that the $\{\omega_n\}_{n\ge 1}$ are mutually independent \footnote{This is of course not true when if $\omega_n$ is distributed as $\tilde{T}^n(x)$, with  $x$ chosen $Leb$-a.e.,  but it becomes asymptotically true since $T$ mixes exponentially fast with respect to $\tilde{\mu}.$}. We thus get a random dynamical system perturbed with additive noise
$$
x_{n+1}=kx_n+(1-k)\omega_n.
$$
It is well known that such random dynamical systems admits a {\em stationary} probability measure $\nu_s.$ For a large class of maps  admitting invariant sets  and   supporting eventually singular measures, the noise has a regularizing  effect making very often the stationary measure absolutely continuous. Moreover the stationary measure is the weak-limit of the sequence
$$
\frac1n\sum_{i=0}^{n-1}\delta_{kx_i+(1-k)\omega_i},
$$
for $\nu_s$ almost all initial condition $x_0$ and almost all realization $\{\omega_n\}_{n\ge 1}$.  Therefore $\nu_s$ could be considered as the weak limit of the sequence of empirical measures  $\nu_n$ constructed in the previous section,\footnote{ Notice that in the limit of zero noise ($k\rightarrow 1$ in our case), the smooth measure $\nu_s$ could converge weakly to an eventual  singular measure and this is coherent with Proposition 3.1. This is a typical weak stochastic stability result.} and therefore the latter   converge to an  absolutely continuous measure. This is confirmed by Fig. 6, which shows the support of the limiting measure of the $\nu_n;$ for $k>0.2$ the histogram is compatible with the presence of a smooth density\footnote{It is however interesting to observe that the empirical measure is not always absolutely continuous, although  it continues to converge weakly to the physical measure $\mu$ of the master map. This is shown in Fig. 7.}.  

\begin{figure}[htbp]
    \centering
    \includegraphics[height=2.6in]{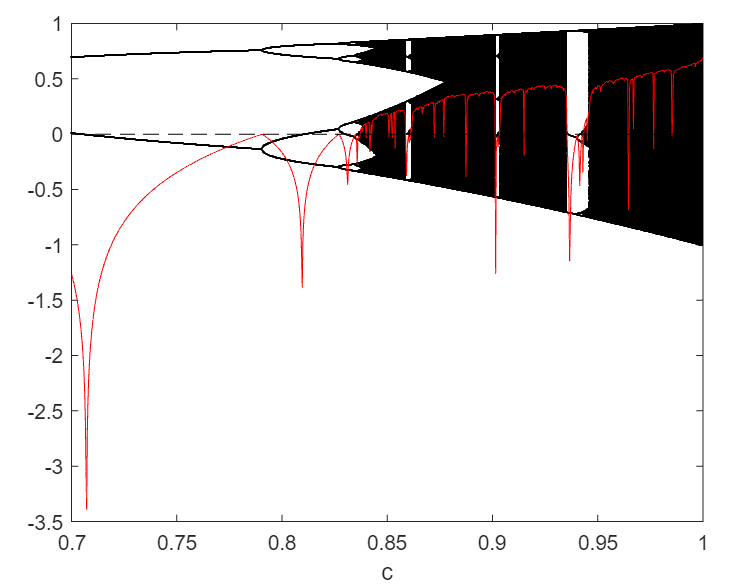}
    \caption{Bifurcation diagram for the map $T(x)=c(1-2x^2)$ and its associated $\lambda_T$ computed over different values of the parameter $c$.}
    \label{bif}
\end{figure}

To gauge the convergence of the measure $\nu_n$ to $\mu$, we plot in figure 4  the evolution of the empirical Lyapunov exponent  (we set $\tilde{\lambda}_T$ the limiting value): $$\tilde{\lambda}_T^{n}:=\frac{1}{n}\sum_{i=0}^{n-1}\log|T'(y_i)|, \ y_{i}=(1-k)c_2(1-2y_{i-1}^2)+c_1k(1-2x_{i-1}^2), i\ge 1
,$$ with respect to the master parameter $c_1.$


For the  values of $c_1$ and $c_2$ prescribed in \cite{TS}, the dependence of $\tilde{\lambda}_T$ vs $k$ is made explicit in Fig. 5. We see that when $k\rightarrow 1$ the empirical Lyapunov exponent $\tilde{\lambda}_T$ converges to $\lambda_T.$  

This supports the conclusions of Proposition 3.1, although in principle it could not be applied to $\log|T'|,$ which is not even bounded on $[-1,1].$   \\
In Fig. \ref{d}, we represent the densities associated with the measure $\nu_n$ at $k=0$ and $k=0.5$ and $k=1$ (at which $\nu_n=\mu)$. For all these values of $k$, the density seems to have singularities on a large set of points, which may be compatible with the simple formal models presented in section \ref{mss}, and therefore  with the findings of a non trivial $D_q$ spectrum, as found in \cite{TS}.

\section{Conclusions}
The paper \cite{TS} used the spectrum of the generalized dimensions to follow the process of synchronisation in master/slave systems. We showed that for the parameter values of the quadratic map  considered in the aforementioned paper, the master map has an absolutely continuous invariant measure and the attractor is not a Cantor set. We did not find in the literature  any result on the multifractal spectrum of such a measure. We instead gave examples of densities allowing a multifractal structure. In those cases the function $q\rightarrow D_q$ is continuous but not smooth, which is not what was observed in \cite{TS}, unless smoothness was a consequence of numerical approximations. Moreover our  examples suggest that the dimensions are constant for negative $q,$ since the invariant densities are bounded away from zero,  which supports the presence of the zipper effect. \\
We presented a detailed study of the Lyapunov exponent and we believe that it is a much more reliable technique, besides to be more theoretically founded,  to describe the synchronisation process.
\section*{Acknowledgments}
The research of SV was supported by the project {\em Dynamics and Information Research Institute} within the agreement between UniCredit Bank and Scuola Normale Superiore di Pisa and by the Laboratoire International Associé LIA LYSM, of the French CNRS and  INdAM (Italy).
MG is partially supported by G.N.A.M.P.A..
TC was partially supported by CMUP, which is financed by national funds through FCT – Fundaçao para a
Ciência e Tecnologia, I.P., under the project with reference UIDB/00144/2020.\\
TC thanks Dylan Bansard-Tresse for diverse discussions and recommendations, and Jorge M. Freitas for his comments and his expertise on unimodal maps.



\begin{figure*}[t!]
    \centering
    \begin{subfigure}[t]{0.5\textwidth}
        \centering
        \includegraphics[height=2.5in]{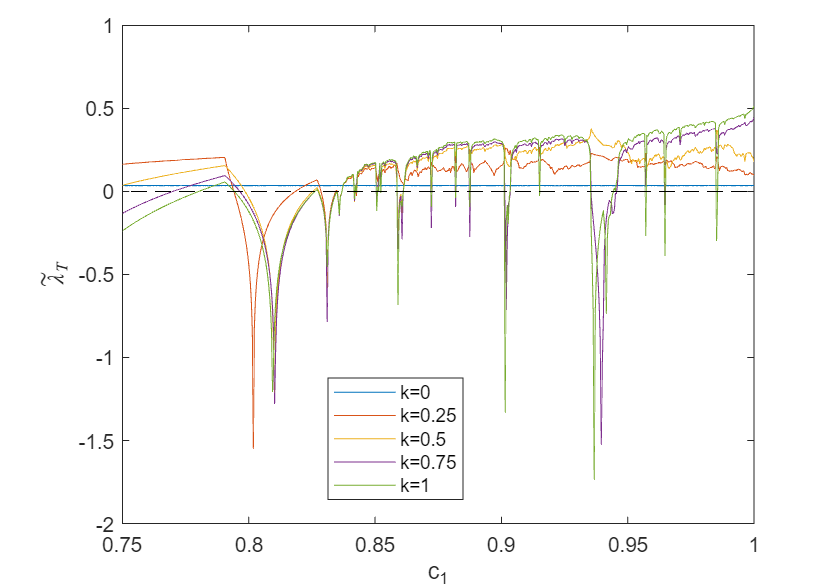}
        \caption{}
    \end{subfigure}%
    ~ 
    \begin{subfigure}[t]{0.5\textwidth}
        \centering
        \includegraphics[height=2.5in]{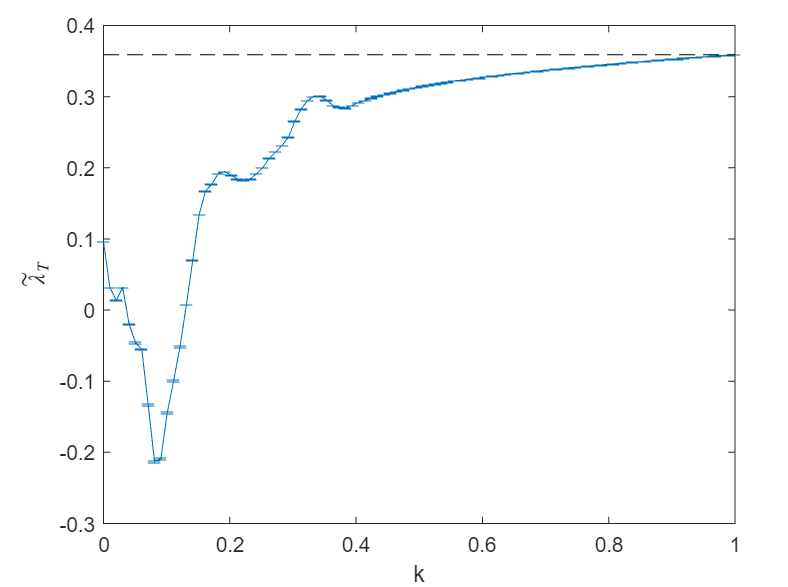}
        \caption{}
    \end{subfigure}
    \caption{Evolution of $\tilde{\lambda}_T$ with $c_1$ for different values of $k$ (left) and with $k$ for the fixed value of $c_1=0.89$ (right). For both figures, we took for $c_2=0.8373351$.}
\end{figure*}

\begin{figure*}[t!]
    \centering
    \begin{subfigure}[t]{0.5\textwidth}
        \centering
        \includegraphics[height=2.5in]{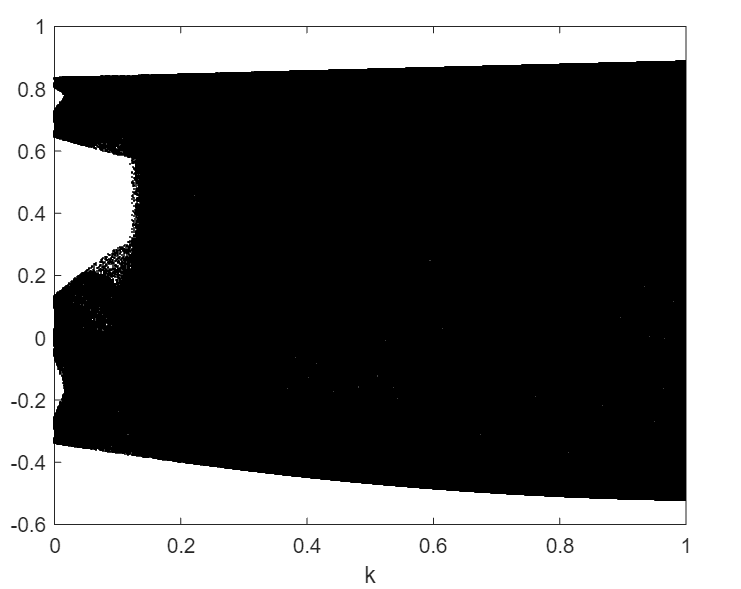}
        \caption{}
    \end{subfigure}%
    ~ 
    \begin{subfigure}[t]{0.5\textwidth}
        \centering
        \includegraphics[height=2.5in]{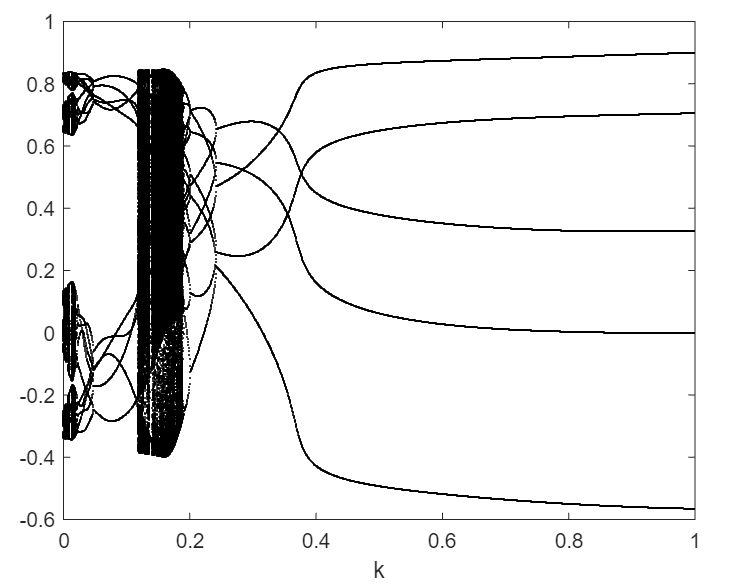}
        \caption{}
    \end{subfigure}
    \caption{Bifurcation diagram of the dynamics of $y_n$ with the parameter $k$, for $c_2=0.8373351$ and $c_1=0.89$ (left) and $c_1=0.901502$ (right). On the left figure, the master measure is supported on an interval, and on the right on a set of 5 points.}
\end{figure*}

\begin{figure}
     \centering
     \begin{subfigure}[b]{0.3\textwidth}
         \centering
         \includegraphics[height=1.7in]{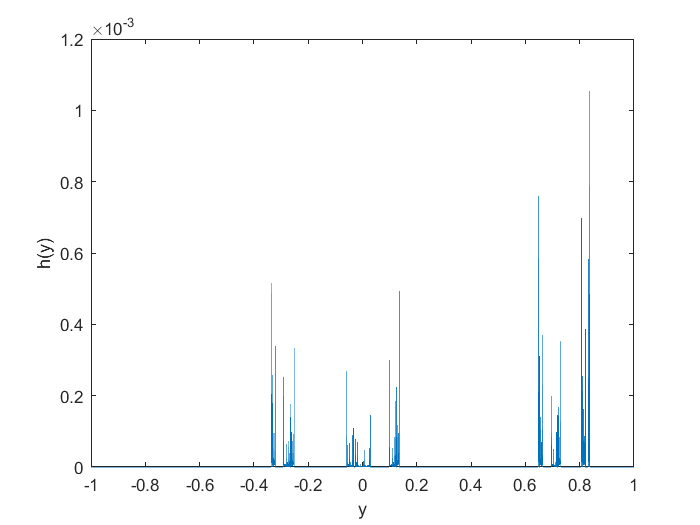}
         \caption{}
         \label{d3}
     \end{subfigure}
     \hfill
     \begin{subfigure}[b]{0.3\textwidth}
         \centering
         \includegraphics[height=1.7in]{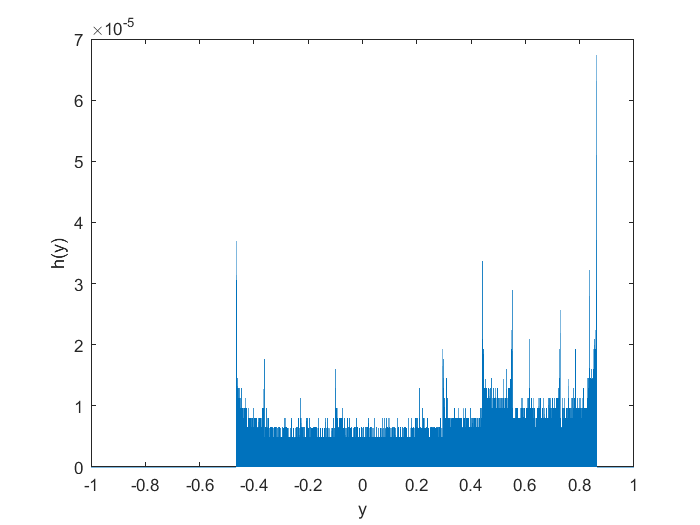}
         \caption{}
         \label{d2}
     \end{subfigure}
     \hfill
     \begin{subfigure}[b]{0.3\textwidth}
         \centering
         \includegraphics[height=1.7in]{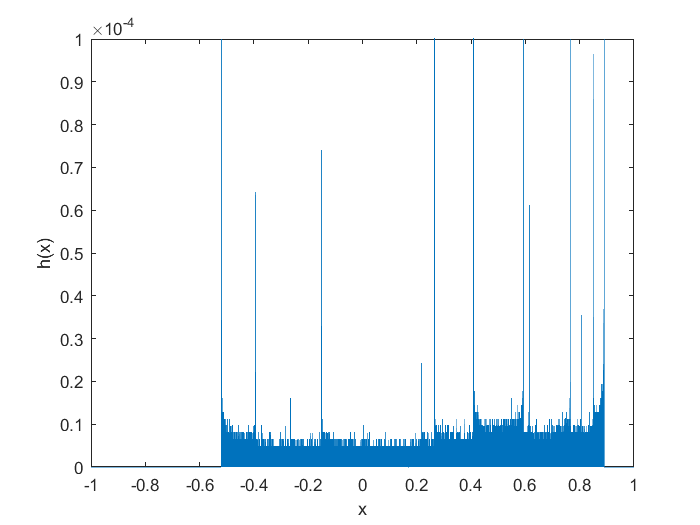}
         \caption{}
         \label{d1}
     \end{subfigure}
        \caption{Numerical estimation of the density of the measure $\nu$ for $k=0$ (left), $k=0.5$ (center) and $k=1$ (right).}
        \label{d}
\end{figure}

\end{document}